\documentclass[reqno,12pt,letterpaper]{amsart}
\usepackage{amsmath,amssymb,amsthm,graphicx,mathrsfs,url,bbm,array,enumerate}
\usepackage[usenames,dvipsnames]{xcolor}
\usepackage[colorlinks=true,linkcolor=Red,citecolor=Green]{hyperref}
\usepackage{tikz-cd}
\usepackage{tikz}
\usepackage{pgfplots}
\usepackage{subfigure}
\usepackage{mathabx}
\usepackage{appendix}

\def\arXiv#1{\href{http://arxiv.org/abs/#1}{arXiv:#1}}

\newcolumntype{L}{>{$}l<{$}}

\def\?[#1]{\textbf{[#1]}\marginpar{\Large{\textbf{??}}}}
\newtheorem*{thm*}{Theorem}
\newtheorem{prop}{Proposition}
\newtheorem{thm}[prop]{Theorem}
\newtheorem{defi}[prop]{Definition}
\newtheorem{lem}[prop]{Lemma}
\newtheorem{cor}[prop]{Corollary}

\numberwithin{equation}{section}
\numberwithin{prop}{section}

\theoremstyle{definition}

\renewcommand{\Re}{\mathop{\rm Re}\nolimits}
\renewcommand{\Im}{\mathop{\rm Im}\nolimits}

\DeclareMathOperator{\Op}{Op}

\DeclareMathOperator{\supp}{supp}

\DeclareMathOperator{\comp}{comp}

\DeclareMathOperator{\WF}{WF}

\newcommand{\Dcal}{{\mathcal D}}
\newcommand{\Ecal}{{\mathcal E}}
\newcommand{\Fcal}{{\mathcal F}}

\newcommand{\Hcal}{{\mathcal H}}

\newcommand{\Ocal}{{\mathcal O}}

\newcommand{\Mcal}{{\mathcal M}}
\newcommand{\Pcal}{{\mathcal P}}
\newcommand{\Kcal}{{\mathcal K}}
\newcommand{\Rcal}{{\mathcal R}}

\newcommand{\Xcal}{{X_1}}

\newcommand{\RR}{{\mathbb R}}
\newcommand{\CC}{{\mathbb C}}

\newcommand{\Ss}{{\mathbb S}}

\newcommand{\NN}{\mathbb{N}}

\newcommand{\HH}{\mathbb{H}}
\newcommand{\El}{{\rm ell}}

\newcommand{\td}{\widetilde}

\begin{document}

\title[Spectral gap]{Spectral gap for surfaces of infinite volume with negative curvature}

\author{Zhongkai Tao}
\email{ztao@ihes.fr}
\address{Institut des Hautes \'{E}tudes Scientifiques, 91440 Bures-sur-Yvette, France}

\begin{abstract}
We prove that the imaginary parts of scattering resonances for negatively curved asymptotically hyperbolic surfaces are uniformly bounded away from zero and provide a resolvent bound in the resulting resonance-free strip. This provides an essential spectral gap without the pressure condition. This is done by adapting the methods of \cite{nsz11}, \cite{vasy2013a}, and \cite{vacossin2022spectral} and answers a question posed in \cite{dyatlov2016zahl}.
\end{abstract}

\maketitle

\section{Introduction}
In a seminal paper, Bourgain--Dyatlov \cite{bourgain2018spectral} showed that a convex cocompact hyperbolic surface has an essential spectral gap between the unitarity axis and the set of scattering resonances. This means the Selberg zeta function has only finitely many zeros in the region $\Re s>1/2-\beta$ for some $\beta>0$. This holds without any assumptions on the Hausdorff dimension of the trapped set, in particular without a ``pressure condition" which in this case goes back to the works of Patterson and Sullivan \cite{patterson1976limit,Sullivan1979}. The purpose of this note is to generalize this result to negatively curved surfaces which are asymptotically hyperbolic in a sense described below. This is done by combining the quantum monodromy method of Nonnenmacher--Sj\"ostrand--Zworski \cite{nsz11} and Vasy's method for meromorphic continuation \cite{vasy2013a,zworski2016vasymethod,resbook} with the recent work of Vacossin \cite{vacossin2022spectral}. It answers a question posed by Dyatlov--Zahl \cite{dyatlov2016zahl} in the first paper on the fractal uncertainty principle.

Let $X$ be an even asymptotically hyperbolic manifold. This means that $X$ has a compactification $\overline{X}$, which is a manifold with smooth boundary $\partial X$, and the metric on $X$ near the boundary takes the form
\begin{equation}\label{e:even-defi}
    g=\frac{dx_1^2+g_1(x_1^2)}{x_1^2},\quad x_1|_{\partial X}=0,\quad dx_1|_{\partial X}\neq 0
\end{equation}
where $g_1(x_1^2)$ is a smooth family of metrics on $\partial X$. See \cite[\S 5.1]{resbook} for a discussion of the invariance of this definition. Let $\Delta$ be the (negative) Laplacian on $X$. We prove
\begin{thm*}
    Suppose $X$ has dimension $2$ and (strictly) negative curvature. Then there exist $C_0,\beta>0$ such that the resolvent
    \begin{equation}
        R(\lambda)=(-\Delta-1/4-\lambda^2)^{-1}:L^2_{\rm comp}(X)\to L^2_{\rm loc}(X)
    \end{equation}
    continues holomorphically from $\Im \lambda> 1$ to $\{|\lambda|>C_0,\, \Im \lambda>-\beta\}$. Moreover, for any $\chi\in C_c^\infty(X)$, we have the resolvent bound 
    \begin{equation}\label{e:res-bound}
        \|\chi(-\Delta-1/4-\lambda^2)^{-1}\chi\|_{L^2\to L^2}\leq C|\lambda|^{-1-C_1\min(0,\Im \lambda)}\log|\lambda|
    \end{equation}
    for $ \Im \lambda>-\beta$ and $|\lambda|>C_0$.
\end{thm*}
The proof of the main Theorem follows from \cite{vacossin2022spectral} by reducing the problem to quantum monodromy maps using \cite{nsz11}. Although \cite{vacossin2022spectral} also uses \cite{nsz11}, our approach is different by replacing the application of microlocal weight functions with propagation estimates. This approach simplifies some aspects of \cite{nsz11} and allows a seamless application of Vasy's method \cite{vasy2013a}. The geometric component comes from the now classical work of Eberlein \cite{eberlein1972geodesic} which shows the trapped set is topologically one dimensional in our setting (see \S 2.3).

The spectral gap for open hyperbolic quantum systems has been studied since Ikawa \cite{ikawa1988decay} in mathematics and Gaspard--Rice \cite{gaspard1989scattering} in physics -- if the topological pressure (an object from thermodynamical formalism) satisfies $P(1/2)<0$, the statement of the theorem above holds with $\beta<-P(1/2)$. For an experimental manifestation of this gap, see \cite{barkhofen13}. A general spectral gap under the pressure condition was proved by Nonnenmacher--Zworski \cite{nonnenmacherzw2009quantum,nonnenmacherzw2009semiclassical}. The first advances in the direction of improving the pressure gaps were made by Naud \cite{Naud} (in the setting of constant curvature surfaces and complex dynamics) and Petkov--Stoyanov \cite{Petkov} (in the setting of obstacle scattering). These results were based on Dolgopyat's method \cite{Dol}. However, it was conjectured by Zworski \cite[Conjecture 3]{zworski2017mathematical} that the pressure condition is not necessary. Dyatlov--Zahl \cite{dyatlov2016zahl} made the first breakthrough showing a spectral gap without the pressure condition by introducing the \textit{fractal uncertainty principle}. Bourgain--Dyatlov \cite{bourgain2018spectral} proved the fractal uncertainty principle for any porous set in dimension one and showed that any (noncompact) convex cocompact hyperbolic surface has an essential spectral gap. For recent advances on the fractal uncertainty principle, see \cite{BLT,cohen-cantor,Cohen}. The spectral gap was generalized to open quantum maps in dimension $2$ by Vacossin \cite{vacossin2022spectral,vacossin2023resolvent} by combining the method of \cite{bourgain2018spectral} and \cite{dyatlov2022control}. That provided an essential spectral gap for classes of obstacles and semiclassical scattering problems.

The bound \eqref{e:res-bound} was recently used in Huang--Sogge--Tao--Zhang \cite{HSTZ} to show lossless Strichartz estimates and spectral projection estimates in the same setting.

\subsection*{Acknowledgement}
We thank Maciej Zworski for many helpful discussions, and for his encouragement to write this note. We thank St\'ephane Nonnenmacher for helpful discussions on Vacossin's paper \cite{vacossin2022spectral}. We thank Semyon Dyatlov for helpful discussions on the appendix, especially for explaining the compactness argument to me. We also thank Jialun Li and Wenyu Pan for the example of rigid acylindrical hyperbolic manifolds. The author was partially supported by the NSF grant DMS-1952939 and by the Simons Targeted Grant Award No. 896630.

\section{Microlocal preliminaries}

\subsection{Review of concepts} We use the terminology of \cite[Appendix E]{resbook}. For the reader's convenience, we review some concepts.

Let $X$ be a smooth manifold without boundary (not necessarily compact). The polyhomogeneous symbols $S^m_h(T^*X)$ of order $m$ on $T^*X$ are defined in \cite[Definition E.3]{resbook} and their quantizations are semiclassical pseudodifferential operators $\Psi_h^m(X)$ on $X$ defined in \cite[Definition E.12]{resbook}. For $a\in S^m_h(T^*X)$ there is a non-canonical construction of $\Op_h(a)\in\Psi_h^m(X)$, see \cite[Proposition E.15]{resbook}. Conversely, for $A\in \Psi_h^m(X)$, there is a canonical principal symbol $\sigma_h(A)\in S^m(T^*X)$ defined in \cite[Proposition E.14]{resbook}. 

We will consider $h$-tempered distributions on $X$, that is $u_h\in \Dcal'(X)$ such that for any $\chi\in C_c^\infty(X)$, there exists $N$ with $\|\chi u\|_{H_h^{-N}}\leq Ch^{-N}$. For an $h$-tempered distribution, the semiclassical wavefront set $\WF_h(u)\subset \overline{T}^*X$ is defined as the complement of the union of open sets $U\times V$ in $\overline{T}^*X$ such that 
\begin{equation*}
    \Fcal_h(\chi u)(\xi)=\Ocal(h^\infty\langle \xi\rangle^{-\infty}),\quad \chi\in C_c^\infty(U), \xi\in V\cap\RR_\xi^d,
\end{equation*}
where
\begin{equation}
    \Fcal_h u(\xi):=\frac{1}{(2\pi h)^{d/2}}\int_{\RR^d} u(x) e^{-ix\cdot \xi/h} dx
\end{equation}
is the semiclassical Fourier transform. The semiclassical wavefront set of an operator $A:C_c^\infty(Y)\to \Dcal'(X)$ with $h$-tempered Schwartz kernel $\Kcal_A$ is defined as (see \cite[Definition E.36]{resbook})
\begin{equation*}
    \WF_h'(A):=\{(x,\xi;y,\eta)\in \overline{T}^*(X\times Y): (x,\xi;y,-\eta)\in \WF_h(\Kcal_A)\}.
\end{equation*}
There is also a notation of wavefront set of a pesudodifferential operator $A\in \Psi_h^m(X)$, defined in \cite[Definition E.27]{resbook}, which has the property that if $A=\Op_h(a)$ then $\WF_h(A)$ is the complement of the union of open sets $W\in \overline{T}^*X$ such that
\begin{equation*}
    a(x,\xi)=\Ocal(h^\infty\langle \xi\rangle^{-\infty}),\quad (x,\xi)\in W\cap T^*X.
\end{equation*}
An operator $C_c^\infty(X)\to \Dcal'(X)$ is  called compactly supported if its Schwartz kernel is compactly supported. We define the compactly microlocalized operators $\Psi^{\rm comp}_h(X)$ as compactly supported operators $A\in\Psi^{m}_h(X)$ such that $\WF_h(A)$ is a compact subset of $T^*X$, see \cite[Definition E.28]{resbook}. For two $h$-tempered distributions $u,v\in \Dcal'(X)$, we say $u=v$ microlocally in some open subset $U\subset \overline{T}^*X$ if $\WF_h(u-v)\cap U=\varnothing$.

If $X$ is compact, we define the semiclassical Sobolev space $H_h^s(X)$ by the norm
\begin{equation*}
    \|u\|_{H_h^s}^2=\sum\limits_j\|\langle hD\rangle^s \varphi_j^*(\chi_j u) \|_{L^2}^2,
\end{equation*}
where $\chi_j$ is a partition of unity and $\varphi_j:\RR^d\supset U\to V\subset X$ are coordinate charts such that $\supp\chi_j\subset V$. When $X$ is not compact, then we can similarly define the local semiclassical Sobolev space $H_{h,{\rm loc}}^s(X)$ and the compactly supported semiclassical Sobolev spaces $H_{h,\comp}^s(X)$ as in \cite[Definition E.20]{resbook}.

If $\overline{X}$ is a compact manifold with boundary $\partial X$ and interior $X$, we embed it into a compact manifold $X_{\rm ext}$ without boundary, and define $\bar{H}_h^s(X)$ by restrictions of functions in $H_h^s(X_{\rm ext})$.

Finally we recall the microlocal estimates in \cite[Thereom E.33, E.47]{resbook}. For $A\in\Psi_h^m(X)$, the elliptic set $\El_h(A)$ is defined as $\{(x,\xi)\in \overline{T}^*X: \langle \xi\rangle^{-m}\sigma_h(A)(x,\xi)\neq 0\}$.
\begin{prop}\label{p:propagation-general}
    Let $P\in \Psi_h^m(X)$ be properly supported such that $\Im \langle\xi\rangle^{-m}\sigma_h(P)\leq 0$.
    \begin{enumerate}
        \item  Suppose $A,B_1\in\Psi_h^0(X)$ are compactly supported and $\WF_h(A)\subset \El_h(P)\cap \El_h(B_1)$. Then there exists $\chi\in C_c^\infty(X)$ such that for any $N$,
        \begin{equation}\label{e:ell-estimate}
            \|Au\|_{H_h^s}\leq C\|BPu\|_{H_h^{s-m}}+\Ocal(h^\infty)\|\chi u\|_{H_h^{-N}}.
        \end{equation}
        \item Suppose $A,B,B_1\in\Psi_h^0(X)$ are compactly supported, and for any $(x,\xi)\in\WF_h(A)$ there exists $T\geq 0$ such that for $p=\Re\sigma_h(P)$,
        \begin{equation}\label{e:prop-estimate-general}
            \exp({-T\langle \xi\rangle^{1-m}H_p})(x,\xi)\in \El_h(B),\quad \exp({-t\langle \xi\rangle^{1-m}H_p})(x,\xi)\in \El_h(B_1) \text{ for all } t\in [0,T].
        \end{equation}
        Then there exists $\chi\in C_c^\infty(X)$ such that for any $N$,
        \begin{equation*}
            \|Au\|_{H_h^s}\leq C\|Bu\|_{H_h^s}+Ch^{-1}\|B_1 Pu\|_{H_h^{s-m+1}}+\Ocal(h^\infty)\|\chi u\|_{H_h^{-N}}.
        \end{equation*}
    \end{enumerate}
\end{prop}
We will also use the sharp Gårding inequality \cite[Proposition E.23]{resbook}.
\begin{prop}\label{p:garding}
    If $A\in \Psi^{2m+1}_h(X)$ is compactly supported and $\Re\sigma_h(A)\geq 0$. Then
    \begin{equation}\label{e:garding}
        \Re\langle Au,u\rangle_{L^2}\geq -Ch\|u\|_{H_h^{m}}^2.
    \end{equation}
\end{prop}

\subsection{Vasy's method revisited}
Vasy \cite{vasy2013a} provided a very general method for showing meromorphic continuation of the resolvent for asymptotically hyperbolic systems. One can also look at \cite{zworski2016vasymethod} for an elementary introduction to Vasy's method and Dyatlov--Zworski \cite[Chapter 5]{resbook} for a more detailed presentation. We start by recalling \cite[Theorem 5.30, Theorem 5.33]{resbook}.

\begin{prop}\label{p:vasy}
    Let $X$ be an even asymptotically hyperbolic manifold of dimension $d+1$ with negative curvature, then there exist a compact manifold $\Xcal$ with boundary $\partial\Xcal$, containing $X$ as an open subset, and a second-order semiclassical differential operator $P(z)$ on $\Xcal$ with the following properties.
    \begin{itemize}
        \item For $z\in[-h,h]+i[-C_0h,Ch]$ and $s>C_0+1/2$,
        \begin{equation}\label{e:fredholm-map}
            P(z): D_h^s=\{u\in \bar{H}_h^s(\Xcal): P(0)u\in \bar{H}_h^{s-1}\}\to \bar{H}_h^{s-1}(\Xcal)
        \end{equation}
        is a holomorphic Fredholm family of index $0$.
    \item The set of poles of $P(z)^{-1}$ contains the set of poles of $(h^2(-\Delta-d^2/4)-(1+z)^2)^{-1}$ (with multiplicity).
    \end{itemize}
\end{prop}

We sketch the proof of Proposition~\ref{p:vasy} and refer the readers to \cite[Chapter 5]{resbook} for more details and to \cite{semibook} for preliminaries on semiclassical analysis. Recall the metric on $X$ is of the form \eqref{e:even-defi} with some boundary defining function $x_1$. First one needs to change the smooth structure on $\overline{X}$ so that $\mu=x_1^2$ becomes a boundary defining function. Then we conjugate the operator $h^2\left(-\Delta-d^2/4\right)-(1+z)^2$ to another operator $P(z)$, which has the form 
\begin{equation}\label{e:Pz-definition}
    P(z)=\mu^{-1-d/4+i(z+1)/2h}\left(h^2\left(-\Delta-\frac{d^2}{4}\right)-(1+z)^2\right)\mu^{d/4-i(z+1)/2h}
\end{equation}
near the boundary $\partial X$. Then $P(z)$ is well-defined with smooth coefficients up to the boundary, and we can extend it over the boundary to some slightly larger manifold $\Xcal$. The Fredholm property of $P(z)$ follows from the  propagation estimates and radial estimates, see \cite[\S 5.5]{resbook}.

From the construction of $P(z)$ and the propagation estimates, we have the following properties. We recall the trapped set is $K_0:=\Gamma_+\cap \Gamma_-$ where the outgoing/incoming sets are defined as
\begin{equation}\label{eq:inout}
    \Gamma_{\pm}:=\{(x,\xi)\in T^*X\setminus 0: \exp(tH_{p})(x,\xi) \text{ remains bounded as }t\to \mp \infty\}\cap p^{-1}(0).
\end{equation}
\begin{itemize}
    \item (\cite[Theorem 5.34]{resbook} and Proposition~\ref{p:propagation}) There exists $Q\in \Psi_{h}^{\rm comp}(X)$ such that $P(z)-ihQ:D_h^s\to \bar{H}_h^{s-1}(\Xcal)$ is invertible for $0<h<h_0$, with the bound
        \begin{equation}\label{e:res-modi-bound}
            \|(P(z)-ihQ)^{-1}\|_{\bar{H}_h^{s-1}\to \bar{H}_h^s}\leq Ch^{-1}.
        \end{equation}
    \item (\cite[\S 5.3]{resbook}) $P(z)$ has real principal symbol. For any pre-fixed neighbourhood $V_0$ of the trapped set $K_0$, we can require $\WF_h(Q)\subset V_0$.

\item (\cite[Theorem 5.35]{resbook} and Proposition~\ref{p:propagation}) Let $p=\sigma_h(P(0))$ be the principal symbol and $\varphi^t=\exp(t\langle \xi\rangle^{-1}H_p)$, then
    \begin{equation}\label{e:wf-res}
        \WF_h'((P(z)-ihQ)^{-1})\cap \overline{T}^*(X\times X)\subset \Delta_{\overline{T}^*X}\cup \Omega_+\cup \Omega_{\Gamma}
    \end{equation}
    where $\Delta_{\overline{T}^*X}:=\{(x,\xi,x,\xi):(x,\xi)\in \overline{T}^*X\}$, 
    \begin{equation*}
        \Omega_+:=\{(\varphi^t(y,\eta),y,\eta): (y,\eta)\in T^*X,\, p(y,\eta)=0,\, t\geq 0\}
    \end{equation*}
    is the positive flowout and $\Omega_\Gamma=\Gamma_+\times\Gamma_-$.
\end{itemize}
Note we use $(P(z)-ihQ)^{-1}$ instead of $(P(z)-iQ)^{-1}$ to get a slightly better estimate in \eqref{e:forward-solve-hQ}. This works because of the following Lemma.
\begin{lem}\label{l:hQ}
    Let $P\in \Psi^m_h(X)$ and $Q\in \Psi_h^{\rm comp}(X)$. Suppose $\Im \sigma_h(P)\leq 0$ and $\Re \sigma_h(Q)>0$ near a compact set $K\subset T^*X$. Then there exists $Y_1,Y_2,Z \in \Psi_h^{\rm comp}(X)$ and $M>0$ such that 
    \begin{equation*}
        Y_1Y_2=I+\Ocal(h^{\infty}) \text{ near } K,\quad K\subset \El_{h}(Z)
    \end{equation*}
    and we have the following estimate for any $N>0$ with some $\chi\in C_c^\infty(X)$ and sufficiently small $h>0$:
    \begin{equation*}
        \Im\langle Y_1(P-iMhQ)Y_2u, u\rangle_{L^2}\leq -h\|Zu\|_{L^2}^2+\Ocal(h^\infty)\|\chi u\|_{H_h^{-N}}^2.
    \end{equation*}
\end{lem}
We remark that similar modifications are also used in \cite[Proposition 2.7]{jin2023number}.
\begin{proof}
    Let $Z\in \Psi_h^{\rm comp}$ be a microlocal cutoff to a neighbourhood of $K$. Since $\Im\sigma_h(P)\leq 0$ near $\WF_h(Z)$, by \eqref{e:garding} we have for some constant $C>0$,
    \begin{equation*}
        \Im\langle P Zu, Zu\rangle_{L^2}\leq Ch\|Zu\|_{L^2}^2+\Ocal(h^\infty)\|\chi u\|_{H_h^{-N}}^2.
    \end{equation*}
    By assumption $\sigma_h(Q)>c>0$ near $\WF_h(Z)$, by \eqref{e:garding} we have
    \begin{equation*}
        \Re\langle QZu, Zu\rangle_{L^2}\geq c\|Zu\|_{L^2}^2-Ch\|Zu\|_{L^2}^{2}-\Ocal(h^\infty)\|\chi u\|_{H_h^{-N}}^2.
    \end{equation*}
    Consequently, 
    \begin{equation*}
         \Im\langle (P-iMhQ) Zu, Zu\rangle_{L^2}\leq (C-cM)h\|Zu\|_{L^2}^2+\Ocal(h^2)\|Zu\|_{L^2}^2+\Ocal(h^\infty)\|\chi u\|_{H_h^{-N}}^2.
    \end{equation*}
    Taking $Y_1=Z^*$, $Y_2=Z$ and $Mc>C+10$ finishes the proof.
\end{proof}

Using Lemma~\ref{l:hQ}, we have the following estimate similar to \cite[Theorem 5.34]{resbook} but with $Q$ replaced  by $MhQ$.
\begin{prop}\label{p:propagation}
    Let $Q\in \Psi_h^{\rm comp}(X)$ such that $\sigma_h(Q)\geq 0$ everywhere and $\sigma_h(Q)>0$ near the trapped set $K_0$. Then for sufficiently large $M>0$, $0<h<h_0$, $z\in[-h,h]+i[-C_0h,Ch]$ and $s>C_0+1/2$, we have
    \begin{equation}\label{e:propagation}
        \|u\|_{\bar{H}^s_h(\Xcal)}\leq Ch^{-1}\|(P(z)-iMhQ)u\|_{\bar{H}_h^{s-1}(\Xcal)}.
    \end{equation}
\end{prop}
\begin{proof}
    First, by \cite[Lemma 5.25]{resbook}, there exists $\chi_1\in C_c^\infty(\Xcal)$ such that
    \begin{equation*}
         \|u\|_{\bar{H}^s_h}\leq Ch^{-1}\|(P(z)-iMhQ)u\|_{\bar{H}_h^{s-1}}+C\|\chi_1 u\|_{H^s_h}.
    \end{equation*}
    The phase space dynamics on $\overline{T}^*\Xcal$ can be described as follows (see \cite[\S 5.4]{resbook}). There exist $\Sigma_\pm$ such that $\{\langle \xi\rangle^{-2}p=0\}\subset \overline{T}^*\Xcal$ is the disjoint union of $\Sigma_+$ and $\Sigma_-$. Moreover, $\Sigma_+\cap \overline{T}^*X=\varnothing$. For $(x,\xi)\in \Sigma_\pm$, we have two possibilities 
    \begin{itemize}
        \item $\varphi^t(x,\xi)\to L_\pm$ as $t\to \pm \infty$, where $L_\pm=\{\mu=0\}\cap \Sigma_\pm\cap\partial \overline{T}^*\Xcal$ are the radial sets.
        \item $\varphi^{t}(x,\xi)\to K_0$ as $t\to -\infty$.
    \end{itemize}
    By propagation estimates (Proposition~\ref{p:propagation-general}), it suffices to estimate near $L_\pm$ and $K_0$. Near $L_\pm$ we use the radial estimate \cite[Lemma 5.23]{resbook}. Near $K_0$ we use our Lemma~\ref{l:hQ} which gives for some microlocal cutoff $A$ to a neighbourhood of $K_0$ (see \cite[Lemma 2.7]{open})
    \begin{equation*}
        \|Au\|_{H_h^s}\leq C\|A_1u\|_{H_h^s}+Ch^{-1}\|(P(z)-ihQ)u\|_{\bar{H}_h^{s-1}}+Ch^{1/2}\|A_2u\|_{H_h^{s-1/2}}+\Ocal(h^\infty)\|\chi u\|_{H_h^{-N}},
    \end{equation*}
where $A_1$ has the property that $\varphi^t(\WF_h(A_1))\to L_-$ as $t\to-\infty$, $A_2\in \Psi_h^{\comp}(X)$ and $\chi\in C_c^\infty(X)$. Then the $A_1u$ term can be propagated to $L_-$ and the $A_2u$ term can be improved to $h^N\|\chi u\|_{H_h^{-N}}$ by iterating the estimate. 
\end{proof}
The resolvent bound \eqref{e:res-modi-bound} follows from \eqref{e:propagation}. The wavefront set estimate \eqref{e:wf-res} follows from the proof of Proposition~\ref{p:propagation}. In order to show \eqref{e:wf-res}, we need to show for any $(x_0,\xi_0;y_0,\eta_0)\in \overline{T}^*(X\times X)$ such that $(x_0,\xi_0)\notin (\Delta_{\overline{T}^*X}\cup \Omega_+\cup \Omega_{\Gamma})(y_0,\eta_0)$, there are open neighbourhoods $U$ of $(x_0,\xi_0)$ and $V$ of $(y_0,\eta_0)$ such that
\begin{equation}\label{e:wf-res-proof}
    \langle (P(z)-ihQ)^{-1}Bu, Av\rangle=\Ocal(h^\infty)\|u\|_{H_h^{-N}}\|v\|_{H_h^{-N}}
\end{equation}
 for any compactly supported $A,B\in\Psi_h^0(X)$ with $\WF_h(A)\subset U$, $\WF_h(B)\subset V$ and any $u,v\in C_c^\infty(X)$. By elliptic estimate \eqref{e:ell-estimate} we may assume $p(x_0,\xi_0)=p(y_0,\eta_0)=0$. Note $T^*X\cap\{p=0\}\subset\Sigma_-$, so $\varphi^{t}(x_0,\xi_0)\to L_- \text{ or } K_0$ as $t\to -\infty$. If $\varphi^t(x_0,\xi_0)\to L_-$ as $t\to -\infty$, we conclude \eqref{e:wf-res-proof} from the propagation estimate \eqref{e:prop-estimate-general}. If $\varphi^t(x_0,\xi_0)\to K_0$ as $t\to -\infty$, then $(x_0,\xi_0)\in \Gamma_+$. A dual estimate would then give \eqref{e:wf-res-proof} unless $(y_0,\eta_0)\in\Gamma_-$.

Now we state a forward solvability property (up to $\Ocal(h^\infty)$ error) which will be crucial to our analysis.
\begin{prop}\label{p:outgoing}
    Suppose $f\in H_{h,\comp}^{s-1}(X)$ and  $\WF_h(f)\cap \Gamma_-=\varnothing$. Then there exists $u\in \bar{H}^s_h(\Xcal)$ such that $P(z)u=f+\Ocal_{H_{h,\comp}^N}(h^\infty)$ for any $N$, and
    \begin{equation*}
        \|u\|_{\bar{H}^s_h}\leq Ch^{-1}\|f\|_{\bar{H}^{s-1}_h},\quad \WF_h(u)\cap \overline{T}^*X\subset\bigcup\limits_{t=0}^{\infty}\varphi^t(\WF_h(f)).
    \end{equation*}
\end{prop}
\begin{proof}
    Since $\WF_h(f)\cap \Gamma_-=\varnothing$, we may choose a small microlocal cutoff $Q\in\Psi_h^{\rm comp}$ to a neighbourhood of $K_0$ so that the backward flow of $\WF_h(Q)$ does not intersect $\WF_h(f)$. Define $u=(P(z)-ihQ)^{-1}f$, then it suffices to show $Qu=\Ocal_{H_{h,\comp}^N}(h^\infty)$. By \eqref{e:wf-res}, $\WF_h(u)\cap \overline{T}^*X\subset \bigcup\limits_{t=0}^{\infty}\varphi^t(\WF_h(f))\cup \Omega_\Gamma\circ \WF_h(f)$. Since $\WF_h(f)\cap\Gamma_-=\varnothing$, we have $\Omega_\Gamma\circ \WF_h(f)=\varnothing$. Moreover, $\bigcup\limits_{t=0}^{\infty}\varphi^t(\WF_h(f))$ does not intersect with $\WF_h(Q)$, thus $\WF_h(u)\cap\WF_h(Q)=\varnothing$ and $Qu=\Ocal_{H_{h,\comp}^N}(h^\infty)$.
\end{proof}

\section{Quantum monodromy without weights}

In order to apply \cite{vacossin2022spectral} to obtain the spectral gap, we need to construct quantum monodromy maps as in \cite{nsz11} but in the asymptotically hyperbolic setting. In \S\ref{s:grushin} we apply Vasy's method to construct such quantum monodromy maps without using weights for $d+1$-dimensional asymptotically hyperbolic manifolds whose trapped set has topological dimension $1$. In \S\ref{s:structure} we use \cite{eberlein1972geodesic} to verify this assumption for all asymptotically hyperbolic surfaces.

\subsection{Construction of quantum monodromy maps}\label{s:grushin}
In this section, we reduce the problem to a quantum monodromy map following \cite{nsz11}. In particular we show the following.
\begin{prop}\label{p:quant-mono}
    Suppose we are in the setting of Proposition~\ref{p:vasy} and the trapped set $K_0$ has topological dimension $1$. Then there exists a holomorphic family of matrices $M(z,h)$ acting on $\CC^N$ with $N\sim h^{-d}$ for $z \in [-h,h]+i[-C_0h,Ch]$ so that the poles of $P(z)^{-1}$ are given by the zeros of $\det(I-M(z,h))$.
\end{prop}

The proof of this proposition follows from the construction of a Grushin problem. This construction proceeds in two steps. First one constructs a \textit{microlocal} Grushin problem near the trapped set $K_0$, which is done in \cite{nsz11} and we can directly use it here. The second step is to construct a \textit{global} Grushin problem. This is done in \cite[\S 5]{nsz11} using weight functions. Here we apply a different construction by using propagation estimates alone. This allows a simple connection with Vasy's method and gives better bounds (see the remark after \cite[Corollary 1]{vacossin2023resolvent}).

\subsubsection{Dynamical preliminaries}
Suppose the trapped set $K_0$ is topologically one dimensional. Then by \cite{bowen1972expansive} (for 
 the statement here we cite \cite[Proposition 2.1]{nsz11}), there exist finitely many compact contractible smooth (up to boundary) hypersurfaces $\Sigma_j\subset p^{-1}(0)$ so that
\begin{itemize}
    \item $\partial\Sigma_j\cap K_0=\varnothing$, $\Sigma_j\cap \Sigma_{j'}=\varnothing$, $j\neq j'$;
    \item $H_p$ is transversal to $\Sigma_j$ up to the boundary;
    \item The Hamiltonian flow starting from $K_0$ touches $\cup_j\Sigma_j$ in both directions. Moreover, we may assume the successor of a point in $\Sigma_k\cap K_0$ is in $\Sigma_j$ for some $j\neq k$.
\end{itemize}
We recall some notations from \cite{nsz11}.
Let $\mathcal{T}=K_0\cap 
\cup_j \Sigma_j$ be the reduced trapped set, and $\mathcal{T}_j=K_0\cap \Sigma_j$. Let $f:\mathcal{T}\to \mathcal{T}$ be the Poincar\'e map restricted to $\mathcal{T}$ (see \cite[\S 2.3.1]{nsz11}) and
\begin{equation*}
    \mathcal{D}_{jk}=\mathcal{T}_k\cap f^{-1}(\mathcal{T}_j),\quad \mathcal{A}_{jk}=\mathcal{T}_j\cap f(\mathcal{T}_k)
\end{equation*}
be the \textit{departure} and \textit{arrival} sets (note $f(\mathcal{D}_{jk})=\mathcal{A}_{jk}$). We take (disjoint) neighbourhoods $D_{jk}\subset \Sigma_k$ ($A_{jk}\subset \Sigma_j$, respectively) of $\mathcal{D}_{jk}$ ($\mathcal{A}_{jk}$, respectively) and extend $f$ to a local symplectomorphism 
\begin{equation*}
    F_{jk} : D_{jk}\to A_{jk}.
\end{equation*}
We may assume $D_{jk}$ and $A_{jk}$ are mutually disjoint and denote
\begin{equation*}
    D_k=\bigcup\limits_{j} D_{jk},\quad A_{j}=\bigcup\limits_{k}A_{jk}.
\end{equation*}
As in \cite{nsz11}, we may choose $\td{\Sigma}_j\subset T^*\RR^{d}$ and smooth symplectomorphisms $\kappa_j: \td{\Sigma}_j\to \Sigma_j$ up to the boundary (by taking $\Sigma_j$ with small diameter). We then define $\td{\mathcal{T}}_j$, $\td{D}_{jk}$, $\td{A}_{jk}$ and $\td{F}_{jk}$ accordingly using $\kappa_j$'s. Now the dynamics is encoded by the monodromy map $F=(F_{jk})$. One quantity that will be useful later is the minimal propagation time
\begin{equation}\label{e:minimal-prop-time}
    t_0:=\min\{t>0:\text{ there exist } j\neq k \text{ and } (x,\xi)\in \Sigma_k \text{ such that } \varphi^t(x,\xi)\in\Sigma_j \}>0.
\end{equation}

\subsubsection{Microlocal Grushin problem}
We now recall the microlocal Grushin problem constructed in \cite[\S 4]{nsz11}. Let $H(V)$ be the space of functions microlocalized in $V$. We will always assume $z\in [-h,h]+i[-C_0h,Ch]$. 

\begin{lem}\label{l:grushin-micro} There exist neighbourhoods $V_0\Subset V_1\Subset V_2\Subset X$ of $K_0$, and semiclassical Fourier integral operators $\td{R}_-^j:L^2(\RR^{d})\to H(V_2)$, $\td{R}_+^j:L^2(X)\to H(\td{\Sigma}_j)$ such that for any $v$ microlocalized in $V_1$, and any $v_+^k$ microlocalized in $\td{D}_k$, we can find $u$ microlocalized in $V_2$, and $u_-^k$ microlocalized   in $\td{D}_k\cup \td{A}_k$, so that $(u,u_-)$ solves
\begin{equation*}
    \begin{pmatrix}
        \frac{i}{h}P(z) &\td{R}_-\\
        \td{R}_+ &0
    \end{pmatrix}\begin{pmatrix}
        u\\
        \, u_-
    \end{pmatrix}=\begin{pmatrix}
        v\\
        \, v_+
    \end{pmatrix}
\end{equation*}
microlocally inside $V_1\times (\bigtimes_k\td{D}_k)$.
\end{lem}
A more precise description of $\td{R}_\pm$ is given as follows.
\begin{itemize}
    \item $\WF_h'(\td{R}_+^j)\subset
    \td{\Sigma}_j\times V_2$ and $\WF_h'(\td{R}_-^j)\subset V_2\times \td{\Sigma}_j$.
    \item For any prefixed $\epsilon>0$, we can require
    \begin{equation}\label{e:wf-r}
\begin{aligned}
        \WF_h'(\td{R}_+^j)&\subset \{(x,\xi; \varphi^t(\kappa(x,\xi))): (x,\xi)\in \td{\Sigma}_j, |t|<\epsilon\},\\
    \WF_h'(\td{R}_-^j)&\subset \{(\varphi^t(\kappa(x,\xi)) ;x,\xi) : (x,\xi)\in \td{\Sigma}_j , t>-\epsilon\}.
\end{aligned}
\end{equation}
\end{itemize}
The microlocal Grushin problem is solved microlocally by
\begin{equation*}
    u=\td{E}v+\td{E}_+v_+,\quad u_-=\td{E}_-v+\td{E}_{-+}v_+
\end{equation*}
where $\td{E}, \td{E}_\pm$ and $\td{E}_{-+}$ are compactly supported operators with compact wavefront sets in $T^*(X\times X)$ and the following almost forward propagation properties:
\begin{equation}\label{e:wf-e}
    \begin{aligned}
        \WF_h'(\td{E})&\subset \{(\varphi^t(x,\xi) ;x,\xi) : t>-\epsilon\},\\
        \WF_h'(\td{E}_-^j)&\subset \{(x,\xi; \varphi^t(\kappa_j(x,\xi))) : (x,\xi)\in \td{\Sigma}_j , t<\epsilon\},\\
        \WF_h'(\td{E}_+^j)&\subset \{(\varphi^t(\kappa_j(x,\xi)) ;x,\xi) : (x,\xi)\in \td{\Sigma}_j , t>-\epsilon\},\\
        \td{E}_{-+}^{jk}&=\Mcal_{jk}-\delta_{jk}, 
    \end{aligned}
\end{equation}
where $\Mcal_{jk}$ is a Fourier integral operator quantizing $\td{F}_{jk}$.

\subsubsection{Global Grushin problem}
The next step is to construct a global Grushin problem. Let $V^j\Subset \kappa_j^{-1}(V_1\cap\Sigma_j)$ be small neighbourhoods of $\td{\mathcal{T}}_j$ in $\td{\Sigma}_j$, $Q_0^j$ be the (self-adjoint) quantization of a cutoff which is positive in $V^j$ and negative outside $\overline{V}^j$, and $\Pi_j$ be the orthogonal projection defined by $\mathbf{1}_{>0}(Q_0^j)$. Denote $V=\cup_j V^j$. Consider the following Grushin problem.
\begin{equation}\label{e:grushin-global}
    \Pcal(z)
    \begin{pmatrix}
        u\\
        \, u_-
    \end{pmatrix}:=\begin{pmatrix}
        \frac{i}{h}P(z) &R_-\\
        R_+ &0
    \end{pmatrix}\begin{pmatrix}
        u\\
        \, u_-
    \end{pmatrix}=\begin{pmatrix}
        v\\
        \, v_+
    \end{pmatrix}:D^s_h(\Xcal)\times \Hcal_h\to \bar{H}^{s-1}_h(\Xcal)\times \Hcal_h.
\end{equation}
where 
\begin{equation}
    \Hcal_h=\bigtimes_j \Hcal_h^j, \quad \Hcal_h^j= \Pi_j H(\td{D}_j),\quad R_-^j=\td{R}_-^j\Pi_j,\quad R_+^j=\Pi_j\td{R}_+^j.
\end{equation}
Note $\Hcal_h$ is a finite dimensional space with $\dim \Hcal_h\sim h^{-d}$. Recall $\WF_h(Q)\subset V_0$ and $\WF_h(\Pi_j)\subset\overline{V}$.
We will choose $V_0$ and $V$ to be sufficiently small neighbourhoods in the following to conclude the well-posedness of the Grushin problem.
\begin{lem}\label{l:grushin-global}
    There exists 
    \begin{equation}\label{e:grushin-global-Ecal}
        \mathcal{E}(z)=\begin{pmatrix}
        E& E_+\\
        E_-&E_{-+}
    \end{pmatrix}=\Ocal(1):\bar{H}^{s-1}_h(\Xcal)\times \Hcal_h\to D^s_h(\Xcal)\times \Hcal_h
    \end{equation}
    solving the Grushin problem \eqref{e:grushin-global}, i.e. $\Pcal(z)\Ecal(z)=I$.
\end{lem}
\begin{proof}[Proof of Lemma~\ref{l:grushin-global}]
Since $\Pcal(z)$ is a Fredholm operator of index $0$, it suffices to construct a right inverse of $\Pcal(z)$. In other words, given $(v,v_+)\in \bar{H}_h^{s-1}(\Xcal)\times \Hcal_h$ with $\|v\|_{\bar{H}_h^{s-1}}^2+\|v_+\|_{L^2}^2\leq 1$, we want to find $(u,u_-)\in D_h^s(\Xcal)\times \Hcal_h$ such that 
\begin{equation*}
    \Pcal(z)\begin{pmatrix}
        u\\
       \, u_-
    \end{pmatrix}=\begin{pmatrix}
        v\\
       \, v_+
    \end{pmatrix},\quad \|u\|_{\bar{H}_h^{s}}^2+\|u_-\|_{L^2}^2=\Ocal(1).
\end{equation*}

\noindent\textbf{Step 1: Microlocalize to the trapped set.} 

In order to apply the microlocal Grushin problem in Lemma~\ref{l:grushin-micro}, we need to first microlocalize to a neighbourhood of the trapped set. Thus we take $$u_0=\frac{h}{i}(P(z)-ihQ)^{-1}v$$ so that
\begin{equation}\label{e:forward-solve-hQ}
    \frac{i}{h}P(z)u_0=P(z)(P(z)-ihQ)^{-1}v=v+ihQ(P(z)-ihQ)^{-1}v.
\end{equation}
Let 
\begin{equation*}
    v_0=-ihQ(P(z)-ihQ)^{-1}v,\quad v_{0+}=R_+u_0,
\end{equation*}
we aime to solve for $\Pcal(z)(u,u_-)^T=(v_0,v_+ - v_{0+})^T$.
At this point we can use the microlocal Grushin problem to find $(\td{u}, \td{u}_-)\in H(V_2)\times H(\td{D}\cup \td{A})$ such that 
\begin{equation*}
    \frac{i}{h}P(z)\td{u}+\td{R}_-\td{u}_-=v_0 - f,\quad R_+\td{u}=v_+-v_{0+}+\Ocal_{L^2}(h^\infty)
\end{equation*}
where $\WF_h(f)\subset V_2\setminus V_1$. 

\noindent
\textbf{Step 2: Correct  the error $f$ using forward propagation.} 

We want to correct the error $f$ without affecting the projection $R_+$. For this we use the forward propagation property in Proposition~\ref{p:outgoing} and the following important observation (see \cite[Lemma 1.4]{open}):
\begin{equation}\label{e:incoming-avoid}
\begin{aligned}
    &\text{Suppose } U \text{ is a neighbourhood of the trapped set } K_0. \\
    &\text{Then there exists a smaller }
    \text{neighbourhood } U_0\subset U \text{ such that}\\ 
    &\text{if }(x,\xi)\in U_0 \text{ and } 
    \varphi^t(x,\xi)\notin U  \text{ for some } t>0,
    \text{ then } (x,\xi)\notin\Gamma_-.
\end{aligned}
\end{equation}

Note \eqref{e:wf-r} and \eqref{e:wf-e} imply that any $(x,\xi)\in \WF_h(f)$ is of the form $\varphi^t(y,\eta)$ for some $(y,\eta)\in V_0\cup \kappa(V)$ and $t>-2\epsilon$. By choosing $V_0$ and $V$ sufficiently small, \eqref{e:incoming-avoid} implies $\WF_h(f)\cap\Gamma_-=\varnothing$. Consequently, by Proposition~\ref{p:outgoing}, there exists $u_1\in D_h^s(\Xcal)$ such that
\begin{equation*}
    \frac{i}{h}P(z) u_1= f+\Ocal_{\bar{H}_h^{s-1}}(h^\infty),\quad \WF_h(u_1)\cap \overline{T}^*X\subset \bigcup\limits_{t=0}^\infty \varphi^t(\WF_h(f)).
\end{equation*}
We then define $\Ecal^{(1)}$ by
\begin{equation*}
    \begin{pmatrix}
        u^{(1)}\\
        u_-^{(1)}
    \end{pmatrix}:=\begin{pmatrix}
        \td{u}+u_1+u_0\\
        \Pi\td{u}_-
    \end{pmatrix}=\begin{pmatrix}
        E^{(1)}&E_+^{(1)}\\
        E_-^{(1)}&E_{-+}^{(1)}
    \end{pmatrix}\begin{pmatrix}
        v\\
       \, v_+
    \end{pmatrix}.
\end{equation*}
Since $u_1$ is the forward solution and $V_0, V$ are taken sufficiently small, we may also assume $R_+u_1=\Ocal_{L^2}(h^\infty)$. Thus
\begin{equation*}
    \frac{i}{h}P(z)u^{(1)}+R_-u_-^{(1)}=v+\td{R}_-(\Pi-I)\td{u}_-+\Ocal_{\bar{H}_h^{s-1}}(h^\infty), \quad R_+ u^{(1)} = v_+ +\Ocal_{L^2}(h^\infty).
\end{equation*}
In other words
 \begin{equation*}
     \Pcal\Ecal^{(1)}=I-\Rcal,\quad \Rcal\begin{pmatrix}
        v\\
        \, v_+
    \end{pmatrix}=\begin{pmatrix}
        \td{R}_-(I-\Pi)\td{E}_-&\td{R}_-(I-\Pi)\Mcal\\
         0&0
     \end{pmatrix}\begin{pmatrix}
        v_0\\
        v_+-v_{0+}
    \end{pmatrix}+\Ocal_{\bar{H}_h^{s-1}\times L^2}(h^\infty).
 \end{equation*}

\noindent
\textbf{Step 3: Correct $\Rcal$ by iteration.}

Finally we need to remove the error $\Rcal$. This is done by showing $\Rcal$ is nilpotent modulo $\Ocal(h^\infty)$. First we note $\Mcal$ has a minimal propagation time $t_0>0$ defined in \eqref{e:minimal-prop-time} which gives
 \begin{equation*}
     \WF_h'(\Mcal)\subset \{(\varphi^t(x,\xi), x,\xi): t
     \geq t_0\}.
 \end{equation*}
Moreover, since $\WF_h'(\td{E}_-Q)\subset V_2\times V_0$, for $V_0$ sufficiently small we have
 \begin{equation*}
     \WF_h'((I-\Pi)\td{E}_-Q)\subset \{(\varphi^t(x,\xi), x,\xi): t
     \geq t_0\}.
 \end{equation*}
So we conclude
\begin{equation*}
    \WF_h'(\Rcal)\subset \{(\varphi^t(x,\xi), x,\xi): t
     \geq t_0-2\epsilon\}.
\end{equation*}
Due to the projection $I-\Pi$, the wavefront set of $\Rcal^N$ does not intersect $\Gamma_-$ for $N\geq N_0$:
\begin{equation*}
     \WF_h'(\Rcal^N)\subset \{(\varphi^t(x,\xi), x,\xi): t
     \geq N(t_0-2\epsilon), (x,\xi)\notin\Gamma_-\},\quad N\geq N_0.
\end{equation*}
 Eventually $\WF_h(\Rcal^{N}(v,v_+)^T)\cap\WF_h(\td{R}_-)=\varnothing$ and thus there exists $N_1\in \NN$ with $\Rcal^{N_1}=\mathcal{O}(h^\infty)$. Let
 \begin{equation*}
     \Ecal^{(2)}:= \Ecal^{(1)}(I+\Rcal+\cdots+\Rcal^{N_1-1}),
 \end{equation*}
then $\Pcal\Ecal^{(2)}=I+\Ocal(h^\infty)$.
So we finally conclude the inverse
\begin{equation}\label{e:grushin-global-Ez}
    \Ecal(z)=\Ecal^{(2)}(I+\Ocal(h^\infty))^{-1}=\Ecal^{(1)}(I+\Rcal+\cdots+\Rcal^{N_1-1})(I+\Ocal(h^\infty)).
\end{equation}
One checks that each step is uniformly bounded in $h$. This finishes the proof of Lemma~\ref{l:grushin-global}.
\end{proof}
\begin{proof}[Proof of Proposition~\ref{p:quant-mono}]

Let $E_{-+}$ be the matrix component defined in \eqref{e:grushin-global-Ecal}.
We define the matrices $M(z,h):=I+E_{-+}(z)$ and the statement follows from the Grushin problem \eqref{e:grushin-global}.
\end{proof} 
We remark that 
$M(z,h)$ has the form $M(z,h)=\Pi\Mcal(z,h)\Pi+\Rcal_1$
where $\Rcal_1$ again satisfies
\begin{equation*}
    \WF_h'(\Rcal_1^N)\subset \{(\varphi^t(x,\xi), x,\xi): t
     \geq N(t_0-2\epsilon), (x,\xi )\notin 
     \Gamma_-\}, \quad N\geq N_0
\end{equation*}
which implies $\Rcal_1^{N_1}=\Ocal(h^\infty)$. Moreover, a direct computation shows that $\Rcal_1$ has the form $\Rcal_1=A(I-\Pi)\Mcal(z,h)\Pi+\Ocal(h^\infty)$ where $A$ satisfies the forward propagation property
\begin{equation*}
    \WF_h'(A)\subset \{(\varphi^t(x,\xi), x,\xi): t
     \geq 0\}.
\end{equation*}

\subsection{Structure of the trapped set}\label{s:structure}
In this section, we verify the dynamical assumption in Proposition~\ref{p:quant-mono}: $K_0$ is topologically one dimensional. From now on we assume $X$ is $2$-dimensional.

\begin{prop}
    Suppose $X$ is a negatively curved (even) asymptotically hyperbolic surface, then the trapped set $K_0$ has topological dimension $1$.
\end{prop}
\begin{proof}
The trapped set $K_0$ we defined before is the same as the trapped set defined by the geodesic flow $H_{\td{p}}$ on $T^*X\setminus 0$ with $\td{p}(x,\xi):=|\xi|^2-1$. So we can use knowledge of negative curved geometry.

Let $\widetilde{X}$ be the universal cover of $X$, then there is a natural compactification $\overline{\widetilde{X}}$, such that the boundary at infinity $\partial_{\infty}\widetilde{X}$ (which is topologically a circle) can be thought as equivalence classes of geodesic rays. The original manifold $X$ is then a quotient of the universal cover $\widetilde{X}$ by a discrete group of isometries $\Gamma$. The \textit{limit set} $\Lambda_\Gamma\subset\partial_{\infty}\td{X}$ is defined as the accumulation points of any orbit of $\Gamma$ in $\td{X}$. The lift $\td{K}_0$ of the trapped set $K_0$ to $S^*\td{X}$ is given by the convex hull of the limit set. Using the Hopf parametrisation, we know $\td{K}_0$ is homeomorphic to $((\Lambda_\Gamma\times \Lambda_\Gamma)\setminus \Delta)\times \RR$.

In order to show $K_0$ has topological dimension $1$, it suffices to show the limit set $\Lambda_{\Gamma}$ is totally disconnected. By \cite[Theorem 2.5]{eberlein1972geodesic}, the limit set is either nowhere dense or the full $\partial_\infty\td{X}$. But the hyperbolic ends of $X$ correspond to intervals in $\partial_\infty\td{X}$ that does not belong to the limit set $\Lambda_\Gamma$. So the second case does not happen, and $\Lambda_\Gamma$ has to be nowhere dense and hence totally disconnected.
\end{proof}

\section{Proof of the spectral gap and the resolvent estimate}
We conclude the proof of the main Theorem in this section.
Let $X$ be an (even) asymptotically hyperbolic surface with (strictly) negative curvature. Then (see \cite{Mazzeo87,guillarmou2005meromorphic,vasy2013a,vasy2013b,zworski2016vasymethod}) the resolvent
\begin{equation}
    R(s)=(-\Delta-1/4-\lambda^2)^{-1}:L^2_{\rm comp}(X)\to L^2_{\rm loc}(X)
\end{equation}
has a meromorphic continuation to $\lambda\in\CC$. After rescaling $\lambda=h^{-1}(1+z)$ we can apply Proposition~\ref{p:vasy} and study the poles of $P(z)^{-1}$ for $z\in [-h,h]+i[-C_0h,Ch]$.

We have reduced the problem to quantum monodromy maps in \S \ref{s:grushin}. In the case of surfaces, \cite[Proposition 4.1]{vacossin2022spectral} shows
\begin{prop}
    There exist $c_0, h_0>0$ and $\gamma>0$ such that for $0<h<h_0$, $c_0\log(1/h)\leq N\leq C\log(1/h)$, we have
    \begin{equation*}
        \|\Mcal(z,h)^N\|_{L^2(\RR)\to L^2(\RR)}\leq Ch^\gamma h^{C_1\min(0,h^{-1}\Im z)}.
    \end{equation*}
\end{prop}
From this we conclude 
\begin{prop}\label{p:power-M}
    There exist $c_1, h_0>0$ and $\gamma>0$ such that for $0<h<h_0$, $c_1\log(1/h)\leq N\leq C\log(1/h)$, we have
    \begin{equation*}
        \|M(z,h)^N\|_{L^2(\RR)\to L^2(\RR)}\leq Ch^\gamma h^{C_1\min(0,h^{-1}\Im z)}.
    \end{equation*}
\end{prop}
\begin{proof}
    Consider
    \begin{equation*}
        M(z,h)^{N+2N_2}=M(z,h)^{N_2}(\Pi\Mcal(z,h)\Pi+\Rcal_1)^NM(z,h)^{N_2}.
    \end{equation*}
    We claim this is equal to $M(z,h)^{N_2}\Mcal(z,h)^N M(z,h)^{N_2}$ modulo $\Ocal(h^\infty)$, if we take $N_2$ sufficiently large but fixed. It suffices to show
    \begin{equation*}
        M(z,h)^{N_2}(\Pi \Mcal(z,h)\Pi)^j\Rcal_1 M(z,h)^{N-j-1+N_2}=\Ocal(h^\infty),\quad j\in\{0,1,\cdots,N-1\}
    \end{equation*}
    and
    \begin{equation*}
    M(z,h)^{N_2}\Mcal(z,h)^{k}(I-\Pi)\Mcal(z,h) M(z,h)^{N-k-1+N_2}=\Ocal(h^\infty),\quad k\in\{1, \cdots, N-1\}.
    \end{equation*}
    This follows from the observation that if $(x,\xi)$ lies outside $V$ (due to the projection $I-\Pi$ and the definition of $\Rcal_1$), then it has be to disjoint from either $\Gamma_+$ or $\Gamma_-$, which implies that either $F^{N_2}(x,\xi)$ or $F^{-N_2}(x,\xi)$ has to escape $V$ for $N_2$ sufficiently large.
\end{proof}
Proposition~\ref{p:power-M} already implies the spectral gap through Proposition~\ref{p:quant-mono}. But we can further estimate the resolvent as below. By the Grushin problem \eqref{e:grushin-global}, we have
\begin{equation}\label{e:grushin-inv}
    P(z)^{-1}=\frac{i}{h}(E-E_+E_{-+}^{-1}E_-)
\end{equation}
where $E,E_-,E_+,E_{-+}$ are all $\Ocal(1)$. Moreover, since $E_{-+}(z)=M(z,h)-I$,
\begin{equation*}
    \|E_{-+}^{-1}(z)\| \leq 1+\|M(z,h)\|+\cdots + \|M(z,h)^{N-1}\|+ \|M(z,h)^{N}\|\|E_{-+}^{-1}(z)\|.
\end{equation*}
For $c_1\log(1/h)\leq N\leq C\log(1/h)$ and $\Im z\geq -C_1^{-1}\gamma h/2$, we conclude
\begin{equation*}
    \|E_{-+}^{-1}(z)\| \leq 2(1+\|M(z,h)\|+\cdots + \|M(z,h)^{N-1}\|).
\end{equation*}
Similar to the proof of Proposition~\ref{p:power-M}, we have for $2N_2\leq j\leq C\log(1/h)$,
\begin{equation*}
    \|M(z,h)^j\|\leq \|M(z,h)^{N_2}\Mcal(z,h)^{j-2N_2}M(z,h)^{N_2}\|+\Ocal(h^\infty)\leq Ce^{-C_1 j\min(0, h^{-1}\Im z)}.
\end{equation*}
Since we have $c_1\log(1/h)$ many terms, we conclude
\begin{equation*}
    \|E_{-+}^{-1}(z)\| \leq C\log(1/h)h^{C_1\min(0,h^{-1}\Im z)}.
\end{equation*}
By \eqref{e:grushin-inv}, we have
\begin{equation*}
    \|P(z)^{-1}\|_{\bar{H}^{s-1}_h\to \bar{H}^s_h}\leq Ch^{-1}\log(1/h)h^{C_1\min(0,h^{-1}\Im z)}.
\end{equation*}
The resolvent bound \eqref{e:res-bound} follows from the definition \eqref{e:Pz-definition} of $P(z)$ and rescaling.


\appendix

\section{Spectral gap for convex cocompact hyperbolic manifolds}
In this appendix, we discuss a special case of the spectral gap problem in higher dimensions -- convex cocompact hyperbolic manifolds, following Cohen's recent breakthrough \cite{Cohen}. We prove the following
\begin{thm}\label{thm:appendix}
    Let $\Gamma\backslash\HH^{d+1}$ be a convex cocompact hyperbolic manifold. Suppose the limit set $\Lambda_{\Gamma}\subset \Ss^d$ does not contain a round circle, then there is an essential spectral gap for the Laplace operator on $\Gamma\backslash\HH^{d+1}$.
    More precisely, there exists $C_0,\beta>0$ such that the resolvent
    \begin{equation*}
        R(\lambda)=(-\Delta-d^2/4-\lambda^2)^{-1}:L^2_{\rm comp}(\Gamma\backslash \HH^{d+1})\to L^2_{\rm loc}(\Gamma\backslash \HH^{d+1})
    \end{equation*}
    continues holomorphically from $\Im \lambda> d/2$ to $\{|\lambda|>C_0,\, \Im \lambda>-\beta\}$, and we have the following resolvent estimate for any $\epsilon>0$:
    \begin{equation}\label{eq:spectral-gap-line-resolvent}
        \|\chi (-\Delta-d^2/4-\lambda^2)^{-1}\chi\|_{L^2\to L^2}\leq C_{\chi,\epsilon} |\lambda|^{-1-2\min(0,\Im\lambda)+\epsilon},\quad  |\lambda|>C_0,\,\, \Im \lambda > -\beta.
    \end{equation}
    Here $\chi\in C_c^{\infty}(\Gamma\backslash \HH^{d+1})$ is a compactly supported cutoff.
\end{thm}
Note that there are non-trivial examples of limit sets that contain a round circle, see e.g. \cite{MMO}. It is an open problem whether such manifolds have an essential spectral gap.

We expect the method of this paper will be helpful to further studies of the variable curvature case for asymptotically hyperbolic manifolds in higher dimensions.

Jin--Zhang--Zhang \cite{JZZ} recently gave a quantitative bound in Theorem~\ref{thm:fup-higher} below and an explicit essential spectral gap for quasi-Fuchsian groups.

\subsection{Convex cocompact hyperbolic manifolds}\label{subsec:horocycle}
Convex cocompact hyperbolic manifolds are manifolds of the form $\Gamma\backslash\HH^{d+1}$ such that the convex core is compact, where $\Gamma$ is a discrete subgroup of isometries. We give a brief review here following \cite{DFG} with a focus on the trapping of its dynamics.

The hyperbolic space $\HH^{d+1}$ has (orientation preserving) isometry group $G=SO(d+1,1)_0$. This can be seen by identifying $\HH^{d+1}$ with the sheet $x_0>0$ of 
\begin{equation}\label{eq:hyperboloid}
    \langle x,x\rangle_{\eta}=x_1^2+\cdots+x_{d+1}^2-x_0^2=-1
\end{equation}
in the Minkowski space $\RR^{d+1,1}$ with metric $\eta=\sum_{1\leq j \leq d+1}dx_j^2-dx_0^2$. Let $\Gamma\subset G$ be a discrete torsion-free subgroup acting on $\HH^{d+1}$ (strictly speaking, one needs to consider the non-orientable case; but we can always pass to an orientable double cover). We compactify $\HH^{d+1}$ using the ball model (with the metric blowing up near the boundary). The \textit{limit set} $\Lambda_{\Gamma}\subset \partial \HH^{d+1}=\Ss^{d}$ is defined as the set of limit points of a $\Gamma$-orbit of a point $\mathfrak{o}\in \HH^{d+1}$. The \textit{convex core} is defined as
\begin{equation*}
   \Gamma\backslash \{x\in \HH^{d+1}: \text{there exists a geodesic } \gamma(t) \text{  such that } \gamma(0)=x \text{ and }\gamma(t)\to \Lambda_{\Gamma} \text{ as } t \to \pm \infty\}.
\end{equation*}
The manifold $\Gamma\backslash \HH^{d+1}$ is called \textit{convex cocompact} if its convex core is compact. Let $\pi: S^*X\to X$ be the projection and $K_0\subset S^*X$ be the trapped set defined as in \eqref{eq:inout}. It is not hard to see the convex core is $\pi(K_0)$ and the manifold is convex cocompact if and only if its trapped set is compact. Moreover, the trapped set has the description that its lift to $S^*\HH^{d+1}$ is given by
\begin{equation*}
    \widetilde{K}_0=\{(x,\xi)\in S^*\HH^{d+1}: e^{tH_p}(x,\xi)\to \Lambda_{\Gamma},\,\, t \to \pm \infty\} \simeq (\Lambda_{\Gamma}\times\Lambda_{\Gamma}\setminus \Delta_{\Lambda_{\Gamma}})\times \RR
\end{equation*}
where $\Delta_{\Lambda_{\Gamma}}$ is the diagonal.

It is also helpful to think about the lift to the frame bundle $F^*X$ with each fiber corresponding to the (positively oriented) orthonormal frames in $T^*X$. We have $F^*X \cong \Gamma\backslash G$ as well as $X \cong \Gamma\backslash G/SO(d+1)$ where $SO(d+1)$ in the maximal compact subgroup in $G$ and $S^*X\cong \Gamma \backslash G/ SO(d)$ where $SO(d)$ embeds into the upper diagonal of $SO(d+1)$. There are several invariant vector fields on $G$ that will be useful later.
\begin{enumerate}
    \item The geodesic vector field
    \begin{equation*}
        Y=\begin{pmatrix}
        0&\cdots &0&0\\
        \vdots&\ddots&\vdots&\vdots\\
        0&\cdots &0 & 1\\
        0&\cdots &1 &0
        \end{pmatrix}.
    \end{equation*}
    \item The horocyclic vector field
    \begin{equation*}
        U_j^{\pm}=\begin{pmatrix}
        0&\cdots &0&\cdots&0&0\\
        \vdots&\ddots&\vdots&\ddots&\vdots&\vdots\\
        0&\cdots&0&\cdots& \pm 1&-1\\
        \vdots&\ddots&\vdots&\ddots&\vdots&\vdots\\
        0&\cdots&\mp 1 &\cdots &0 & 0\\
        0&\cdots&-1 &\cdots &0 &0
        \end{pmatrix},\quad j=1,2,\cdots, d,
    \end{equation*}
    where the nonzero entries are in the $(d+1-j)$-th row and $(d+1-j)$-th column.
\end{enumerate}
We have
\begin{equation}\label{eq:XU-communtator}
    [Y,U_{j}^{\pm}]=\pm U_j^{\pm},\quad [U_j^{+},U_j^{-}]=2Y.
\end{equation}
    Note that $Y$ commutes with rotational vector fields in $\mathfrak{so}(d)$ and consequently descends to $S^*X$. However, $U_{j}^{\pm}$ does not commute with the rotational vector fields in $\mathfrak{so}(d)$ and does not descend to $S^*X$. Since
\begin{equation*}
    [U_i^{\pm}, U_j^{\pm}]=0,\quad [U_j^{\pm},\mathfrak{so}(d)]\in \mathrm{span}(U^{\pm}),\quad [U_i^{\pm},U_j^{\mp}]\in \mathfrak{so}(d),\quad i\neq j,
\end{equation*}
$\mathrm{span}(U^{\pm})$ is invariant under $SO(d)$ and thus descends to $S^*X$.

By \eqref{eq:XU-communtator}, the stable bundle $E_s$ is spanned by projections of $U_j^+$ and the unstable bundle $E_u$ is spanned by projections of $U_j^{-}$. Let $\pi_S: F^*X\to S^*X$ be the projection, we can also write
\begin{equation*}
    \pi_S^*(E_{s/u}) = \mathrm{span}(U^{\pm}) \oplus \mathfrak{so}(d).
\end{equation*}

Now we write down the flows in coordinates. In the hyperboloid model \eqref{eq:hyperboloid}, the sphere bundle $S^*\HH^{d+1}$ can be identified with
\begin{equation*}
    \{(x,\xi)\in \RR^{d+1,1}: \langle x,x\rangle_{\eta}=-1,\,\, \langle x,\xi\rangle_{\eta}=0,\,\, \langle \xi,\xi\rangle_{\eta}=1,\,\, x_0>0\}
\end{equation*}
and the frame bundle $F^*\HH^{d+1}$ can be identifies as
\begin{equation*}
    \{(x,\xi^{1},\cdots,\xi^{d+1}): x\in \HH^{d+1},\,\, \xi^{1},\cdots, \xi^{d+1}\in T^* \HH^{d+1} \text{ is a positively oriented orthonormal basis}\}.
\end{equation*}
The map $\pi_S$ can be identified with $(x,\xi^1,\cdots,\xi^{d+1})\to (x,\xi^1)$.
The geodesic flow $e^{tY}$ acts on $S^*\HH^{d+1}$ by
\begin{equation*}
    e^{tY}(x,\xi)= (x\cosh t+\xi \sinh t,x\sinh t+\xi \cosh t)
\end{equation*}
and acts on $F^*\HH^{d+1}$ by
\begin{equation*}
    e^{tY}(x,\xi^1,\cdots,\xi^{d+1})= (x\cosh t+\xi^1 \sinh t,x\sinh t+\xi^1 \cosh t,\xi^2,\cdots,  \xi^{d+1}).
\end{equation*}
The horocyclic vector field $U_j^{\pm}$ acts on $F^*\HH^{d+1}$ by (here the matrices act on the right)
\begin{equation*}
\begin{split}
    &e^{sU_j^{\pm}}(x,\xi^1,\cdots,\xi^{d+1})\\
    =& \left(x-s\xi^{j+1}+\frac{s^2}{2}(x\pm \xi^1),\xi^1\pm s\xi^{j+1}-\frac{s^2}{2}(\xi^1\pm x),\xi^2,\cdots,\xi^{j+1}\mp s\xi^{1}-s x ,\cdots, \xi^{d+1}\right).
    \end{split}
\end{equation*}
Under the Hopf parametrization (we denote $x=(x',x_0)$ and $\xi=(\xi',\xi_0)$)
\begin{equation*}
\begin{split}
    \Phi: S^*\HH^{d+1}\simeq (\Ss_+^{d}\times\Ss_-^{d}\setminus \Delta_{\Ss^d})\times \RR\\
    (x,\xi) \to \left(\frac{x'+\xi'}{|x'+\xi'|},\frac{x'-\xi'}{|x'-\xi'|}, T(x,\xi) \right)
\end{split}
\end{equation*}
where $T(x,\xi)$ is a parameter on the geodesic $e^{tY}(x,\xi)$, we note 
\begin{equation*}
\begin{split}
    \Phi(\pi_S(e^{sU_j^{+}}(x,\xi^1,\cdots,\xi^{d+1}))) &= \left(\frac{x'+(\xi^1)'}{|x'+(\xi^1)'|}, \frac{x'-(\xi^1)'-2s(\xi^{j+1})'+s^2(x'+(\xi^1)')}{|x'-(\xi^1)'-2s(\xi^{j+1})'+s^2(x'+(\xi^1)')|},T\right);\\
    \Phi(\pi_S(e^{sU_j^{-}}(x,\xi^1,\cdots,\xi^{d+1}))) &= \left(\frac{x'+(\xi^1)'-2s(\xi^{j+1})'+s^2(x'-(\xi^1)')}{|x'+(\xi^1)'-2s(\xi^{j+1})'+s^2(x'-(\xi^1)')|},\frac{x'-(\xi^1)'}{|x'-(\xi^1)'|}, T\right).
\end{split}
\end{equation*}
We claim the nontrivial component  
\begin{equation}\label{eq:curve}
   \frac{x'\mp(\xi^1)'-2s(\xi^{j+1})'+s^2(x'\pm(\xi^1)')}{|x'\mp(\xi^1)'-2s(\xi^{j+1})'+s^2(x'\pm(\xi^1)')|}
\end{equation}
gives a circle segment on the sphere $\Ss^d$. This is because $x\mp \xi^1-2s\xi^{j+1}+s^2(x\pm\xi^1)$ lies in the null cone of the $3$-dimensional space $V_{x,\xi^1,\xi^{j+1}}$ spanned by $x,\xi^1$ and $\xi^{j+1}$, and the claim follows from the fact that the intersection of $\Ss^d$ with the two dimensional plane $ V_{x,\xi^1,\xi^{j+1}}\cap \{x_0=1\}$ is a circle. It is also clear that the curve \eqref{eq:curve} cannot be a single point (otherwise, $x',(\xi^1)', (\xi^{j+1})'$ are linearly dependent, which is impossible).

\subsection{Spectral gap}\label{subsec:FUP-unitary}
We recall the fractal uncertainty principle in the regime of unitarity bound.
\begin{defi}\label{def:porous}
    Let $A$ be a compact set in $\RR^d$. Let $\nu>0$ and $0 \leq \alpha_0<\alpha_1$. \begin{itemize}
        \item We say $A$ is $\nu$-porous from scales $\alpha_0$ to $\alpha_1$, if for any ball $B(x,r)$ of radius $r\in (\alpha_0,\alpha_1)$, there exists a ball $B(x',\nu r)$ of radius $\nu r$ inside $B(x,r)$ such that $A\cap B(x',\nu r)=\varnothing$.
        \item We say $A$ is $\nu$-porous along lines from scales $\alpha_0$ to $\alpha_1$, if for any line segment $\ell$ of length $r\in (\alpha_0, \alpha_1)$, there exists a point $x\in \ell$, such that $ A \cap B(x,\nu r) =\varnothing$.
    \end{itemize}
\end{defi}
If a set is $\nu$-porous ($\nu$-porous along lines) from scales $\alpha_0$ to $\alpha_1$ for some $\nu>0$, we will say it is porous (line porous) from scales $\alpha_0$ to $\alpha_1$.

Bourgain--Dyatlov \cite{bourgain2018spectral} establishes the following fractal uncertainty principle (see \cite{JiZh} for a quantitative version).
\begin{thm}
Let $\nu>0$ and $h\in (0,1)$.
    Suppose $A_h$ and $B_h$ are two compact subsets of $[0,1]$ that are $\nu$-porous from scales $h$ to $1$, then there exist $\beta=\beta(\nu)>0$ and $C=C(\nu)>0$ independent of $h$, such that
    \begin{equation*}
        \|\mathbbm{1}_{A_h}\mathcal{F}_h\mathbbm{1}_{B_h}\|_{L^2(\RR)\to L^2(\RR)}\leq Ch^{\beta}.
    \end{equation*}
\end{thm}

In a recent breakthrough, Alex Cohen \cite{Cohen} proves a fractal uncertainty principle in higher dimensions (see \cite{JZZ} for a quantitative version).
\begin{thm}\label{thm:fup-higher}
Let $v>0$ and $h\in (0,1)$.
    Suppose $A_h$ and $B_h$ are two compact subsets of $[0,1]^d$ that
    \begin{itemize}
        \item $A_h$ is $\nu$-porous from scales $h$ to $1$;
        \item $B_h$ is $\nu$-porous along lines from scales $h$ to 1.
    \end{itemize} 
    Then there exist $\beta=\beta(\nu,d)>0$ and $C=C(\nu,d)>0$ independent of $h$, such that
    \begin{equation*}
        \|\mathbbm{1}_{A_h}\mathcal{F}_h\mathbbm{1}_{B_h}\|_{L^2(\RR^d)\to L^2(\RR^d)}\leq Ch^{\beta}.
    \end{equation*}
\end{thm}
By \cite[Proposition 4.14]{kim}, Cohen's theorem applies to a general non-degenerate phase. By \cite{dyatlov2016zahl} we have the following corollary.
\begin{cor}\label{cor:spectral-gap-line}
    Let $X=\Gamma\backslash \HH^{d+1}$ be a convex cocompact hyperbolic manifold. If the limit set $\Lambda_{\Gamma}$ is line porous, then there is an essential spectral gap, and \eqref{eq:spectral-gap-line-resolvent} holds.
\end{cor}

This leaves the question of which limit sets are line porous. We first give some basic properties of line porosity. 
\begin{lem}\label{lem:line-porous}
Let $A\subset \RR^d$ be a compact set.
\begin{enumerate}
    \item Suppose $A$ is $\nu$-porous along lines from scales $\alpha_0$ to $\alpha_1$, and $F:\RR^d\to \RR^d$ is a diffeomorphism such that
    \begin{equation*}
        C_1^{-1}|x-y|\leq |F(x)-F(y)|\leq C_1|x-y|,\quad |\partial_i\partial_j F(x)|\leq C_2.
    \end{equation*}
    Suppose $0<\alpha_0<\alpha_1<\frac{\nu}{C_1C_2d}$.
    Then $F(A)$ is $\frac{\nu}{2C_1^2}$-porous along lines from scales $C_1\alpha_0$ to $C_1\alpha_1$.
    \item Suppose there exists $\nu>0$, such that any line segment $\ell$ of length $r\in (\alpha_0,\alpha_1)$ with midpoint in $A$ contains a point $x\in \ell$ such that $B(x,\nu r)\cap A=\varnothing$. Then $A$ is $\nu/2$-porous along lines from scales $\alpha_0$ to $\alpha_1$.
\end{enumerate}
\end{lem}
\begin{proof}
    \begin{enumerate}
        \item This is \cite[Lemma 4.10]{kim}.
        \item If $A$ is not $\nu/2$-porous from scales $\alpha_0$ to $\alpha_1$, then there exists a line segment $\ell$ with length $r\in (\alpha_0,\alpha_1)$, such that the middle point $x$ of $\ell$ satisfies $B(x,\nu r/2)\cap A \neq \varnothing$. Suppose $x'\in B(x,\nu r/2)\cap A$. Let $\ell'$ be the line segment with midpoint $x'$ parallel to $\ell$ and of length $r$. Then for any $y\in \ell'$, we have $B(y,\nu r) \cap A \neq \varnothing$, which contradicts the assumption. \qedhere
    \end{enumerate}
\end{proof}
In particular, by Lemma~\ref{lem:line-porous} (1), it makes sense to talk about line porosity on a manifold. We provide the following characterization for line porosity of the limit set $\Lambda_{\Gamma} \subset \Ss^{d}$.
\begin{prop}\label{prop:line-porous-criterion}
    The limit set $\Lambda_{\Gamma}\subset \Ss^d$ is line porous from scales $0$ to $1$ if and only if it does not contain a round circle.
\end{prop}
\begin{proof}
    If the limit set $\Lambda_\Gamma$ contains a circle, then it is not line porous because a circle is not line porous. We would like to show the other direction, i.e., if $\Lambda_{\Gamma}$ is not line porous, then it must contain a circle.

    Let $K_S\subset F^*X$ be the trapped set in the frame bundle. First, we show that the following assumption on the frame bundle implies the line porosity of $\Lambda_{\Gamma}$: there exists a constant $T>0$ such that for any $z\in K_S$, there exists $s \in [-T,T]$ such that
    \begin{equation}\label{eq:assumption-frame}
        e^{sU_{1}^{+}}z \notin K_S.
    \end{equation}
    By compactness of $K_S$, this implies that there exists a uniform constant $R>0$ such that for some $s_1\in [-T,T]$, we have $ B(e^{s_1 U_{1}^{+}}z,R) \cap K_S = \varnothing$. For any point $z_0\in K_S$ and a flow line $\ell_{z_0,s_0}=\{e^{sU_1^{+}} z_0 : -s_0\leq s \leq s_0\}$ with $s_0\in (0,1)$, we may choose $t_0\in\RR$ such that $e^{t_0}s_0=T$. Since $ e^{tY}\circ e^{sU_1^{+}}(z)=e^{e^{-t}sU_1^{+}}\circ e^{tY}(z)$ (note the composition is different from matrix multiplication since the action is on the right), we have
    \begin{equation*}
        e^{-t_0Y}(\ell_{z_0,s_0})= \ell_{e^{-t_0 Y}(z_0),e^{t_0}s_0}=\ell_{e^{-t_0 Y}(z_0),T}.
    \end{equation*}
    By our assumption \eqref{eq:assumption-frame}, there exists $s_1\in [-T,T]$ such that
    \begin{equation*}
        B\left(e^{s_1U_1^{+}}\left(e^{-t_0 Y}(z_0)\right),R\right) \cap K_S =\varnothing.
    \end{equation*}
    In particular, for any $U^+=\sum a_jU_j^+$ and $|a_j|\leq R'$ for some $R'=R'(R)>0$, we have
    \begin{equation*}
        e^{-t_0 Y}\left(e^{e^{-t_0}U^+}\circ e^{e^{-t_0}s_1U_1^{+}}(z_0)\right)=e^{U^+}\circ e^{s_1U_1^{+}}\left(e^{-t_0 Y}(z_0) \right) \notin K_S.
    \end{equation*}
    In other words, $z_1=e^{e^{-t_0}s_1U_1^+}(z_0) \in \ell_{z_0,s_0}$ (note $e^{-t_0}s_1\in [-s_0,s_0]$) satisfies
    \begin{equation}\label{eq:line-porous-proof1}
        e^{\sum a_j U_j^+}(z_1)\notin K_S,\quad |a_j|\leq s_0R' T^{-1}.
    \end{equation}
    Now to show $\Lambda_{\Gamma}$ is line porous, we take $x\in \Lambda_{\Gamma}$. Pick another point $x'\in \Lambda_{\Gamma}$ and consider $y\in S^*\HH^{d+1}$ that corresponds to $(x',x)$ under the Hopf parametrization. More precisely, the geodesic emanating from $y$ converges to $x'$ in the forward direction and converges to $x$ in the backward direction. Let $y_0\in S^*X$ be the projection of $y$ to the quotient. Let $z_0\in F^*X$ be a lift of $y_0$ in the frame bundle and $\ell_{z_0,s_0}$ be a path going through $z_0$. Since $z_0\in K_S$, the above reasoning shows that there is a point $z_1\in \ell_{z_0,s_0}$ satisfying \eqref{eq:line-porous-proof1}. Projecting this back to a neighbourhood of $x$ shows that a circle segment corresponding to $\ell_{z_0,s_0}$ contains a point such that its neighbourhood of size $\simeq s_0$ does not intersect $\Lambda_{\Gamma}$. Note by changing the choice of $x'$ and the lift $z_0$ in the frame bundle, this covers different circles (in particular, circles in different directions, see \eqref{eq:curve}) passing though the point $x$. Since the notion of line porosity is invariant under local uniform change of coordinates (see Lemma~\ref{lem:line-porous}), we conclude that $\Lambda_{\Gamma}$ is line porous from scales $0$ to $1$.
    
    Finally, we show that if \eqref{eq:assumption-frame} does not hold, then $\Lambda_{\Gamma}$ must contain a circle. Suppose \eqref{eq:assumption-frame} does not hold, then there exist $z_n\in K_S$ and $T_n\to \infty$ such that 
    \begin{equation*}
        e^{sU_1^+}z_n \in K_S,\quad s\in[-T_n,T_n].
    \end{equation*}
    Since $K_S$ is compact, there exists a limit point $z\in K_S$ of $\{z_n\}$, such that
    \begin{equation*}
        e^{sU_1^+}z \in K_S,\quad s\in\RR.
    \end{equation*}
    Projecting this to $\Ss^{d}$ gives a circle (see the discussion at the end of \S\ref{subsec:horocycle}).
\end{proof}
Combining Corollary~\ref{cor:spectral-gap-line} and Proposition~\ref{prop:line-porous-criterion}, we prove Theorem~\ref{thm:appendix}.



\end{document}